\documentclass[10pt,namelimits,sumlimits]{ip-journal}
\usepackage{amssymb,amsmath}
\usepackage[mathscr]{eucal}
\usepackage{color}
\usepackage{enumitem}
\usepackage[utf8]{inputenc}
\usepackage[leftcaption]{sidecap}
\usepackage{tikz}
 \usepackage{tikz-cd}
 \usepackage{mathtools}
\usetikzlibrary{matrix,arrows,decorations.pathmorphing}
\usetikzlibrary{positioning}
\usetikzlibrary{matrix,arrows,decorations.pathmorphing, shapes.geometric}
\usepackage[title]{appendix}

\usepackage{psfrag}

\usepackage{epstopdf}
\usepackage{graphicx}
\usepackage{verbatim}
\usepackage{amsmath,amsthm}
\usepackage{graphics}
\usepackage{color}
\usepackage{epsfig}
\usepackage{fullpage}
\usepackage{amssymb,amsmath}
\usepackage[mathscr]{eucal}

\newtheorem{theorem}[equation]{Theorem}
\newtheorem{lemma}[equation]{Lemma}
\newtheorem{prop}[equation]{Proposition}

\newtheorem{definition}[equation]{Definition}
\newtheorem{example}[equation]{Example}

\theoremstyle{remark}
\newtheorem{remark}[equation]{Remark}
\newtheorem{notation}[equation]{Notation}

\numberwithin{equation}{section}

\newcommand{\dbold}{{\mathbf{d}}}
\newcommand{\dboldtilde}{\tilde{\dbold}}
\newcommand{\gtilde}{\tilde{g}}
\newcommand{\rtilde}{\tilde{r}}
\newcommand{\nablatilde}{\widetilde{\nabla}}
\newcommand{\Sph}{\mathbb{S}}
\newcommand{\Rcap}{\mathsf{R}}

\newcommand{\R}{\mathbb{R}}

\newcommand{\ddiv}{\mathrm{div}}

\newcommand{\trace}{\text{tr}}

\newcommand{\fbold}{\boldsymbol{f}}
\newcommand{\phibold}{\boldsymbol{\phi}}

\newcommand{\cunder}{\underline{c}}

\newcommand{\nablabar}{\nabla}
\newcommand{\nablasigma}{\nabla^{\Sigma}}

\newcommand{\Phip}{\Phi_p}
\newcommand{\Psip}{\Psi_p}

\begin{document}

\title[]{Area Bounds for Free Boundary Minimal Surfaces in a Geodesic Ball in the Sphere}
\author[B.~Freidin]{Brian~Freidin}
\author[P.~McGrath]{Peter~McGrath}
\date{}
\address{Department of Mathematics, Brown University, Providence,
RI 02912} \email{bfreidin@math.brown.edu}
\address{Department of Mathematics, University of Pennsylvania, Philadelphia PA 19104} 
\email{pjmcgrat@sas.upenn.edu}
\maketitle
\begin{abstract}
We extend to higher dimensions earlier  sharp bounds for the area of two dimensional free boundary minimal surfaces contained in a geodesic ball of the round sphere.  This follows work of Brendle and Fraser-Schoen in the euclidean case. 
\end{abstract}

\section{Introduction}
\label{S:intro}

A problem of recent interest in geometric analysis is to identify sharp area bounds for free boundary minimal surfaces.  Fraser-Schoen proved \cite[Theorem 5.4]{FS1} any  free boundary $\Sigma^2 \subset B^n$, where $B^n$ is a unit $n$-dimensional euclidean ball, has area at least $\pi$; equality holds precisely when $\Sigma$ is congruent to a disk.  Following a question of Guth, Schoen  conjectured the analogous sharp bound $|\Sigma^k |\geq |B^k|$ for free boundary $\Sigma^k \subset B^n$ of any dimension.  This was later proved by Brendle \cite{Brendle:area}.  In \cite{FM}, the authors proved analogous bounds for free boundary $\Sigma^2$ in certain positively curved geodesic balls, 
%including small balls in $(\R^{3}, e^{-|x|^2/4} \delta)$ -- the space in which self-shrinkers for mean curvature flow are minimal surfaces -- and
including any such ball contained in a hemisphere of the round $\Sph^n$.  In this article we extend results of \cite{FM} to higher dimensions.

\begin{theorem}
\label{Tmain}
Let $\Sigma^k \subset B^n_R$ be a free boundary minimal surface, where $B^n_R\subset \Sph^n$ is a geodesic ball with radius $R\leq \pi/2$, and $k=4$ or $k=6$.  Then $|\Sigma |\geq |B^k_R|$, where $|B^k_R|$ is the volume of a $k$-dimensional geodesic ball of radius $R$.  If equality holds, then $\Sigma$ coincides with some such ball. 
\end{theorem}

Applying the proof of Theorem \ref{Tmain} to a sequence of balls $B^n_R$ as above with radii going to zero in combination with a rescaling argument recovers in the limit (see \ref{Peuclid} for a precise statement) the euclidean bounds in \cite{Brendle:area}, giving another proof of those results in the dimensions above. 

A corollary of the euclidean area bounds mentioned above is that free boundary submanifolds of a euclidean ball satisfy the sharp isoperimetric inequality $|\partial \Sigma^k |^k/ |\Sigma^k |^{k-1} \geq |\partial B^k|^k/ |B^k|^{k-1}$.  The class of minimal submanifolds $\R^n$ for which this sharp isoperimetric inequality is known to hold is relatively small and includes also absolutely area minimizing submanifolds \cite{Almgren} and two-dimensional minimal surfaces with radially connected boundary \cite{ChoeEuc}. 
It would be interesting to know whether the submanifolds considered in Theorem \ref{Tmain} satisfy the sharp spherical isoperimetric inequality.   In dimension 2, Choe-Gulliver \cite[Remark 1]{ChoeGulliverI} have asked more generally whether every minimal surface $\Sigma^2$ contained in a hemisphere of $\Sph^n$ (with no conditions on the boundary) satisfies the sharp $\Sph^2$-isoperimetric inequality $4\pi |\Sigma | \leq |\partial \Sigma|^2 + | \Sigma|^2$. 

A  properly immersed submanifold $\Sigma^k\subset \Omega^n$ in a domain of a Riemannian manifold is a \emph{free boundary minimal submanifold} if $\Sigma$ is minimal, $\partial \Sigma \subset \partial\Omega$, and $\Sigma$ intersects $\partial \Omega$ orthogonally.  Such submanifolds are volume-critical among all deformations which preserve the condition $\partial \Sigma \subset \partial\Omega$.  Free boundary minimal submanifolds have been widely studied in the last decade, and many fundamental questions regarding their existence and uniqueness remain unanswered. 

The proof of Theorem \ref{Tmain} is motivated by Brendle's ingenious approach in \cite{Brendle:area}.
There Brendle applies the divergence theorem to a vector field $W$ with the following properties:
\begin{enumerate}[label=(\roman*).]
\item  $W$ is defined on $B^n\setminus \{ y\}$ and has a prescribed singularity at $y\in \partial B^n$.
\item  $W$ is tangent to $\partial B^n$ along $\partial B^n\setminus \{y\}$.
\item $\ddiv_\Sigma W\leq 1$ for any submanifold $\Sigma^k \subset B^n$. 
\end{enumerate}

In the euclidean setting of \cite{Brendle:area}, $W$ is a sum of a radial field with divergence bounded above by $1$ centered at $0$ and a singular field with nonpositive divergence centered at $y$.  When the dimension $k$ of the submanifold is greater than two, $W$ contains an integral term manufactured to cancel an unfavorable %inner product%
 term arising from the dominant singular part.

Unfortunately, the analogous field  -- even in dimension two -- in the setting of Theorem \ref{Tmain} no longer satisfies (iii).  It turns out however that a judiciously chosen convex combination  of fields -- each of which has divergence bounded above by 1 -- can be arranged which satisfies (i)-(iii).
The singular part is governed by a vector field $Z$ of the form 
\[ 
Z = \Psi_y+ \int_{R}^\pi h(s) \Psi_{\gamma(s)} \, ds,
\] 
where $\Psi_y$ is a dominant term with a singularity at $y$, and the singular integral term (integrated along the geodesic segment $\gamma$ connecting $y$ and the antipode of $B_R$'s center)  is manufactured as in the euclidean case to ensure that $W$ is tangent to the geodesic sphere $\partial B_R$ along $\partial B_R\setminus \{ y\}$.   Idiosyncratic aspects of the formula for the volume $|B_R^k|$ of a $k$-dimensional geodesic ball of radius $R$ in $\Sph^n$ make this term fundamentally more complicated than its counterpart in \cite{Brendle:area}, which has several consequences.  

One such aspect is a structural difference between expressions for $|B_R^k|$ when $k$ is even and when $k$ is odd.   Because of this, we are presently able to propose a scheme to adequately construct $Z$ only for even $k$ (see \ref{dZgen} and \ref{RPsi3}).    A similar dichotomy is present in formulae related to other PDE, for example in the solution of the wave equation on $\R^n$ \cite[Theorems 2.4.2, 2.4.3]{Evans} and in formulas for the heat kernel on hyperbolic space $\mathbb{H}^n$ \cite{Grigoryan} and on the sphere $\Sph^n$ \cite{Nagase}.

 %and because of this, we are only able to fully understand the situation in the dimensions above. 
Another consequence is that it is rather trivial (see \ref{Rhalf}) to prove the sharp bound of Theorem \ref{Tmain}  in the special case when $B^n_R$ is a hemisphere -- one may actually take $W = \Psi_y$ and  $h$ identically zero -- but more challenging to understand the state of affairs for general $R$ and $k$.

 Indeed, when $k$ is an even integer $2j$, $h$ is determined by the solution of an initial value problem associated with a $(j-1)\times(j-1)$ first order linear system of differential equations (see \ref{Lsystem}).  Even for small $j$, the associated $h$ is quite involved -- when $j=2$, for example,
\begin{align*}
h(s) = \cunder \left\{ (1+2\csc^2 R) \sin^3 s + 
 ( \cos^2 s -\cos R \cos s - \frac{1}{3} \sin^2 R)\sin s  \left( \frac{\cot(s/2)}{\cot(R/2)}\right)^{\cos R}\right\},
\end{align*}
where $\cunder: =  (3\cos R \csc^2 R)/(1+3\sin^2 R)$.  By contrast, the appropriate euclidean analogue of $h$ \cite[equation (1)]{Hitotumatu} in dimension $k\geq 3$ is simply $s^{k-3}$.  A key step in our method of proof is to verify that $h$ is nonnegative.  We are able to confirm this for $j=2$ and $j=3$ and thus prove Theorem \ref{Tmain} in dimensions $k=4$ and $k = 6$.  When $k=8$ and for certain values of $R$, numerical computations indicate that $h$ is \emph{not} strictly nonnegative and the method appears to break down. 

The calibration vector field strategy in the sprit of \cite{Brendle:area} appears to be quite flexible and has been used recently by Brendle-Hung \cite{Brendle:area2} to prove a sharp lower bound for the area of a minimal submanifold $\Sigma^k \subset B^n$ passing through a prescribed point $y\in B^n$ (see also \cite{Zhu}).

The approach here is also closely related to work of Choe \cite{ChoeCrelle} and Choe-Gulliver \cite{ChoeGulliverM, ChoeGulliverI} on isoperimetric inequalities for domains on minimal surfaces.  While in that setting the geometric inequalities are favorable in a negative curvature background, in the present context positive ambient curvature is essential (see \ref{Lphidiv}) to the proof of Theorem \ref{Tmain}.  Similar interactions between curvature and geometric inequalities lead to a generalization of the classical monotonicity formula for minimal submanifolds of hyperbolic space $\mathbb{H}^n$ \cite{Anderson}, whereas no monotonicity formula is known for minimal submanifolds of the sphere (see however \cite[Lemma 2.1]{GulliverScott} for a weaker result).

\section{Notation and auxiliary results}
\label{S:notation}

%\subsection*{Radial vector fields}
%\label{ss:warped}
Let $(\Sph^n,  g_{\Sph})$ denote the unit $n$-sphere equipped with the round metric.  Given $p\in M$, we write $\dbold_p$ for the geodesic distance function from $p$ and define a closed geodesic ball about $p$ by 
\begin{align*}
B^n_\delta(p) := \left\{ q\in M : \dbold_p(q) \leq \delta \right\}.
\end{align*}
Given $p\in \Sph^n$, recall that the punctured round unit sphere $\Sph^n\setminus \{-p\}$ is isometric to 
\begin{align}
\label{Ewarped}
\Sph^{n-1}\times [ 0, \pi), \qquad
g = dr \otimes dr + w(r)^2 g_{\Sph},
\end{align}
where $r : = \dbold_p$ and $w(r) := \sin  r$. 
Let $|B^k_R|$ be the area of any geodesic $k$-ball with radius $R$.
%When the base point $p$ is clear from context, we simply write $r = \dbold_p$ and $B_r=B_r(p)$.

Throughout, we fix $R\in (0, \pi/2]$ and a geodesic ball $B^n_R(p)$, which we shall refer to in abbreviated fashion as $B_R$.  Let $\Sigma^k \subset B_R$ be a minimal surface. Let $\nabla$ be the covariant derivative on $\Sph^n$ and $\nablasigma, \ddiv_{\Sigma}$, and $\Delta_{\Sigma}$ respectively be the covariant derivative, divergence, and Laplacian operators on $\Sigma$.  It is convenient  to define $\nabla^\Sigma  r^\perp := \nabla  r- \nabla^\Sigma r$; note that $ \left| \nabla^\Sigma r^\perp\right|^2=1-\left|\nabla^\Sigma r\right|^2$.

\begin{definition}
\label{DI}
Define a function $I_k\in C^\infty \left( [0, \pi]\right)$  by 
\begin{align*}
I_k(r) = \int_0^r w^{k-1}(s)\, d s.
\end{align*}
When the context is clear, we may omit the subscript $k$. 
\end{definition}

\begin{remark}
\label{Rvol}
Note that 
\begin{align*}
|B^k_r| = \int_{B^k_r} dV = \int_{0}^r \int_{\Sph^{k-1}}w^{k-1}(s) d\omega ds = \omega_{k-1} I_k(r), 
\end{align*}
where $\omega_{k-1} := \int_{\Sph^{k-1}} d \omega$ is the euclidean area of the unit $(k-1)$-sphere.
\end{remark}

Theorem  \ref{Tmain} follows from the following  general  argument which shifts the difficulty of  the  problem to the  construction of a vector field  with certain  properties.  

\begin{prop}
\label{Pproof}
Suppose for each $y\in  \partial B_R$, there exists a vector field $W$ on $B_R\setminus \{y\}$ satisfying:
\begin{enumerate}[label=\emph{(\roman*).}]
\item As $\dbold_y  \searrow0$,  $W  = - 2 I(R)\dbold^{1-k}_y \nabla  \dbold_y  +o \left( \dbold_y^{1-k}\right)$.
\item $W$  is tangent to $\partial B_R$ along $\partial B_R\setminus \{ y\} $.
\item  $\ddiv_\Sigma W\leq  1$  for any minimal surface $\Sigma^k \subset B_R$, with  equality only  if $\nabla^\Sigma \dbold_p   = \nabla  \dbold_p$  on $\Sigma$.
\end{enumerate}
Then the conclusion of \ref{Tmain} holds. 
\end{prop}
\begin{proof}
Fix $y\in \partial \Sigma$ and $W$ as above.  From the  divergence theorem, the minimality  of $\Sigma$, and (iii), 
\begin{align*}
\left| \Sigma \setminus B_{\varepsilon}(y) \right| \geq \int_{\Sigma \setminus B_{\varepsilon}(y)} \ddiv_\Sigma W 
= \int_{\partial \Sigma \setminus B_{\varepsilon}(y)} \langle W, \eta \rangle + \int_{\Sigma \cap \partial B_{\varepsilon}(y)} \langle W, \eta \rangle.
\end{align*}
By the free boundary condition, $\eta = \nabla  r$  on $\partial \Sigma$; using (ii) and letting $\varepsilon \searrow 0$, we find
\begin{align*}
|\Sigma| \geq \lim_{\varepsilon \searrow 0} \int_{\Sigma \cap \partial B_{\varepsilon}(y)} \langle W, \eta \rangle.
\end{align*}
On  $\Sigma \cap  \partial B_\varepsilon(y)$,  the free boundary condition implies $\eta  = - \nabla\dbold_y +o(1)$; in  combination with  (i) this implies $\langle W,  \eta \rangle =  2I(R)\varepsilon^{1-k} + o\left( \varepsilon^{1-k}\right)$ on  $\Sigma \cap \partial B_\varepsilon(y)$.  The free boundary condition also implies $|\Sigma \cap \partial B_\varepsilon(y)| =  \frac{\omega_{k-1}}{2}\varepsilon^{k-1}+o(\varepsilon^{k-1})$.  Taking $\varepsilon\searrow 0$, we conclude from the  preceding  that  $|\Sigma| \geq  \omega_{k-1} I(R) = |B^k_R|$, where the last equality follows from Remark \ref{Rvol}.

In the case of equality, (iii) implies that the integral curves in $\Sigma$  of $\nabla^\Sigma  \dbold_p$ are also integral curves in $\Sph^n$ of $\nabla \dbold_p$, namely they are parts of geodesics  passing through  $p$.  It follows that  $\Sigma$ is a geodesic $k$-ball of radius $R$ about $p$.
\end{proof}

\begin{definition}
\label{Dphi}
Define a function $\varphi$ on $[0, \pi)$ by $\varphi(t) = I(t) w^{1-k}(t)$, where we define $\varphi(0)  = 0$.  Given $p\in \Sph^n$, define vector fields $\Phi_p$ and $\Psi_p$ on respectively $\Sph^n\setminus \{-p\}$ and $\Sph^n \setminus \{ p\}$ by 
\begin{align*}
\Phi_p =\left( \varphi \circ \dbold_p\right) \nablabar \dbold_p, \qquad \Psi_p = \Phi_{-p}.
\end{align*}
\end{definition}

\begin{lemma}
\label{Lphidiv}
Given $p\in \Sph^n$, the following hold.
\begin{enumerate}[label=\emph{(\roman*).}]
\item $\ddiv_\Sigma \Phip \leq 1$ and $\ddiv_\Sigma \Psip\leq 1$.
\item $\Psip =\left( \psi \circ \dbold_p\right) \nabla \dbold_p$, where $\psi(r) :=  \frac{I(r) - I(\pi)}{\sin^{k-1} r}$.  If $k = 2j$ is even, moreover
\begin{align*}
\psi(r) =  \sum_{i=1}^{j} \frac{-a_i \sin r}{(1-\cos r)^i}, \quad 
a_1 := \frac{1}{2j-1}, \quad a_{i+1}  :=  \frac{2(j-i)}{2j-(i+1)} a_i, \quad  i=1, \dots, j-1.
\end{align*}
\end{enumerate}
\end{lemma}
\begin{proof}
In this proof, denote $r = \dbold_p$.  Take coordinates for $\Sph^n\setminus \{-p \}$ as in \eqref{Ewarped}.  As in the proof of \cite[Lemma 3]{ChoeCrelle},  $\nabla^2 r = (w'/w) \left( g - dr \otimes dr\right)$.   It follows that $\Delta_\Sigma r =( w'/w) (k- \left| \nabla^\Sigma r \right|^2)$.  Then
\begin{align*}
\ddiv_{\Sigma}\Phi_p &= \varphi' \left| \nabla^\Sigma r\right|^2 + \varphi \Delta_\Sigma r\\
&= \varphi'\left| \nabla^\Sigma r\right|^2 + \varphi \frac {w'}{w} ( k - \left| \nabla^\Sigma r\right|^2)\\
&= \varphi'+ (k-1)\varphi \frac{w'}{w} + \left( \varphi \frac{w'}{w} - \varphi'\right) \left|\nabla^\Sigma r^\perp\right|^2\\
&= 1+ w^{-k}\left({k I w' - w^k}\right) \left|\nabla^\Sigma r^\perp\right|^2\\
 &=1+ w^{-k} \left( k\int_{0}^r (I w')' \, dt - w^k\right) \left|\nabla^\Sigma r^\perp\right|^2\\
&=1+k w^{-k}\left( \int_0^r I w'' \,dt\right) \left|\nabla^\Sigma r^\perp\right|^2.
\end{align*}
(i) follows from this and Definition \ref{Dphi}.  For (ii), denote $r = \dbold_p$ and compute 
\begin{align*}
\Psi_p  %=\left( \varphi \circ \dbold_{-p}\right) \nablabar \dbold_{-p}
=- \varphi(\pi-r) \nablabar r
= \frac{-I(\pi - r)}{\sin^{k}(\pi - r)} \sin r \nabla r.
\end{align*}
The first equality follows by using that $I(r) + I(\pi - r) = I(\pi)$.  Next, note that the general solution to 
\begin{align}
\label{Ey}
(\sin r) y' + k (\cos r ) y=  -1
\end{align}
is $y(r) = (-I_{k}(r)+C) \csc^{k} r$.
 By the change of variable $t(r) = \cos r$, \eqref{Ey} is equivalent to
 \begin{align*}
 (t^2 -1) \frac{du}{dt} + k t u =-1.
 \end{align*}
Assume now $k = 2j$ is even.  Define $u(t) = \sum_{i=1}^j a_i (1-t)^{-i}$ where the $a_i$ are as in the statement of the lemma and compute
\begin{align*}
(t^2-1) \frac{du}{dt} +2jt u &= \sum_{i=1}^j \frac{i (t^2-1) a_i}{(1-t)^{i+1}} + \sum_{i=1}^j \frac{2j t a_i}{(1-t)^i}\\
&= \sum_{i=1}^j \frac{-i (1+t) a_i}{(1-t)^i} + \sum_{i=1}^j \frac{2j a_i}{(1-t)^i} - \sum_{i=1}^j \frac{2j a_i}{(1-t)^{i-1}}\\
&= \sum_{i=1}^j \frac{-2i a_i}{(1-t)^i} + \sum_{i=1}^j \frac{i a_i}{(1-t)^{i-1}}+ \frac{2ja_j}{(1-t)^j} + \sum_{i=1}^{j-1} \frac{2j(a_{i}-a_{i+1})}{(1-t)^i} -2ja_1\\
&=\sum_{i=1}^{j-1} \frac{-2i a_i}{(1-t)^i} + \sum_{i=1}^{j-1} \frac{(i+1)a_{i+1}}{(1-t)^{i}} +a_1+ \sum_{i=1}^{j-1}\frac{2j(a_i-a_{i+1})}{(1-t)^j} -2j a_1\\
%&= 1+ \sum_{i=1}^{j-1} \frac{2i a_i - (i+1) a_{i+1}}{(1+t)^i} + 2j \sum_{i=1}^{j-1} \frac{a_{i+1} - a_i}{(1+t)^i}\\ 
&=  -1,
\end{align*}
where the last step uses the definition of the coefficients $a_i$. Since both $u$ and $I_{k}(r) \csc^{k}r$ are nonsingular at $r=0$,  the conclusion follows. 
 \end{proof}
 
 \begin{remark}
 \label{Rjsmall}
When $j=2$, $a_1 = a_2 = 1/3$; when $j=3$, $a_1 = a_2 = 1/5$ and $a_3 = 2/15$. 
 \end{remark}

%%%%%%%%%%%%%%%%%%%%%%%%%%%%%%%%%%%%%%%%%%%%

Now fix $p\in \Sph^n$ and $x, y \in \partial B_R$, where $B_R := B_R(p)$.  Denote $r: = \dbold_y(x) = \dbold_x(y)$.

\begin{remark}
\label{Rhalf}
In the special case where $B^n_R$ is a hemisphere, the area bound $|\Sigma^k | \geq |B^k_R|$ is rather trivial -- one simply defines $W = \Psi_y$
and verifies that \ref{Pproof}.(i)-(iii) hold: (i) follows from \ref{Lphidiv}.(ii), (ii) follows from the definition of $\Psi_y$ and using that $\left\langle \nabla\dbold_y, \nabla \dbold_p\rangle \right|_{\partial B_R} =0$ since $B^n_R$ is a hemisphere, and (iii) follows from \ref{Lphidiv}.(i).  Note that the preceding holds for all $k$ and not just the values asserted in Theorem \ref{Tmain}.  The bound in this special case can also be anticipated from the following geometric heuristic.  Because $B_R$ is a hemisphere and $\Sigma$ meets $\partial B_R$ orthogonally, the union of $\Sigma$ and its reflection across $\partial B_R$ is a closed minimal submanifold $\hat{\Sigma}$ in $\Sph^n$.  Applying the euclidean monotonicity formula to the cone over $\hat{\Sigma}$ implies  $|\hat{\Sigma}| \geq |\Sph^k |$ and hence $|\Sigma^k | \geq |B^k_R|$ after dividing by two.
\end{remark}

By Remark \ref{Rhalf}, we may henceforth assume $R\in (0, \pi/2)$.  The more general definition of $W$ (see \ref{DWgen}; also \ref{dW} and \ref{D6}) reduces to $\Psi_y$ when $R= \pi/2$. 

\begin{lemma}
\label{Lpsiprod}
$\tan R \left. \langle \Psi_y, \nabla \dbold_p \rangle \right|_x = -\sum_{i=0}^{j-1} a_{i+1} (1-\cos r)^{-i}$.
\end{lemma}
\begin{proof}
Combine \ref{Lphidiv}.(ii) and the spherical law of cosines at vertex $x$ in the geodesic triangle $pxy$,\\
namely $\sin R \sin r \left. \langle \nabla \dbold_y, \nabla \dbold_p\rangle \right|_x = \cos R( 1- \cos r)$.
\end{proof}

\begin{remark}
\label{Rint}
The vector field $W$ \cite[Definition 2.11]{FM} used to prove the two-dimensional version of Theorem \ref{Tmain} is a linear combination of $\Phi_p$ and $\Psi_y$.  In dimension $k = 2j$, $j>1$, $\left. \langle \Psi_y, \nabla \dbold_p \rangle \right|_x$ depends on $r$ by Lemma \ref{Lpsiprod} and consequently no such linear combination satisfies \ref{Pproof}.(ii).  For this reason the definition of $W$ here is more involved. 
\end{remark}

\begin{notation}
\label{Ngamma}
Let $\gamma: [ R, \pi] \rightarrow \Sph^n$ be the minimizing geodesic from $y$ to $-p$ parametrized by arc length.  When there is no risk of confusion, we often write $\dbold_x$ in place of $\dbold_x \circ \gamma$.
\end{notation}

We now outline the proof of Theorem \ref{Tmain}. Our general aim is to construct from $\Psi_y$ a vector field $W$ satisfying Proposition \ref{Pproof}.(i)-(iii).  It turns out that the undesirable inner product property of Lemma \ref{Lpsiprod} can be ameliorated by adding to $\Psi_y$ a term of the form $\int_R^\pi h(s) \Psi_{\gamma(s)}\, ds$ so that the resulting vector field $Z$ has constant inner product with $\nabla \dbold_p$ along $\partial B^n_R \setminus \{ y\}$.  In \ref{DWgen} we define $W$ to be an appropriate linear combination of $\Phi_p$ and $Z$ so that \ref{Pproof}.(ii) holds.   \ref{Pproof}.(i) then follows in a straightforward way after observing (Lemma \ref{LZsing} below) that $\int_R^\pi h(s) \Psi_{\gamma(s)}\, ds$ is less singular than $\Psi_y$ at $y$.  On the other hand, verification of \ref{Pproof}.(iii) requires a detailed understanding of the function $h(s)$ which occupies most of the rest of the paper.  In particular, our argument requires us to prove that $h$ is nonnegative so we can conclude from \ref{Lphidiv}.(i) that $\ddiv_\Sigma Z \le 1 + \int_R^\pi h(s)\, ds$.

%The following lemma is an immediate consequence of the spherical law of cosines.
\begin{lemma}[The spherical law of cosines]
\label{Lsincos}
The following hold.
\begin{enumerate}[label=\emph{(\roman*).}]
\item $\displaystyle{
\sin R \sin \left(\dbold_x\circ \gamma \right) \left. \langle\nabla \dbold_{\gamma(s)}, \nabla \dbold_p \rangle\right|_x =  \cos s - \cos R \cos \left(\dbold_x\circ \gamma\right)}.$
 \item 
$\sin s \sin\left( \dbold_x\circ \gamma\right) \left. \left\langle \nabla \dbold_x, \nabla \dbold_p\right \rangle\right|_{\gamma(s)} = \cos R  - \cos s  \cos\left( \dbold_x \circ \gamma\right).$
\item $\sin s \frac{d}{ds}(1 -  \cos \left( \dbold_x \circ \gamma \right)) = \cos R -\cos s + \cos s(1- \cos\left( \dbold_x \circ \gamma\right))$.
\end{enumerate}
\end{lemma}
\begin{proof}
(i)-(ii) are the spherical law of cosines, applied to vertices $x$ and $\gamma(s)$ of the geodesic triangle $p x \gamma(s)$.  (iii) is just a reformulation of (ii).
\end{proof}
%%%%%%%%%%%%%%%%%%%%%%%%%%%%%%%%%%%%%%%%%%%%%%%%%
%%% Parallel Transport Commented out Below 
%%%%%%%%%%%%%%%%%%%%%%%%%%%%%%%%%%%%%%%%%%%%%%%%%
%%%%%%%%%%%%%%%%%%%%%%%%%%%%%%%%%%%%%%%%%%%%%%%%%
%%% END COMMENT
%%%%%%%%%%%%%%%%%%%%%%%%%%%%%%%%%%%%%%%%%%%%%%%%%

\begin{lemma}
\label{LZsing}
Suppose $u$ is a bounded integrable function on $[R, \pi]$.  Then  
\[ \int_R^\pi u(s) \Psi_{\gamma(s)}  \, ds  = o\left( \dbold_y^{1- k}\right) \quad 
\text{as} \quad 
\dbold_y \searrow 0.
\] 
\end{lemma}
\begin{proof}
In this proof, denote $r = \dbold_x(y)$.  Using \ref{Lphidiv}.(ii), estimate
\begin{align*}
\left| \int_R^\pi  u(s) \Psi_{\gamma(s)} \, ds \right| &\le C \int_R^\pi \frac{1}{(\dbold_x\circ \gamma)^{k-1}}\, ds.
\end{align*}
There exists a constant $c>0$ such that along $\gamma$,
$\dbold_x \circ \gamma  > c ( r + s-R )$.
Therefore, 
\begin{align*}
\left| r^{k-1} \int_R^\pi  h(s) \Psi_{\gamma(s)} \, ds \right| 
&\le \frac{C}{c^{k-1}} \int_R^\pi \left( \frac{r}{r + (s-R)}\right)^{k-1} ds.
\end{align*}
The result now follows from the dominated convergence theorem. 
\end{proof}

\section{Constructing $Z$}
\label{S:Z}

\subsection*{The 4 dimensional case}

\begin{definition}
\label{dW}
Define a vector field $Z$ on $B_R \setminus \{ y\}$ by 
\begin{align*}
Z = \Psi_y+ \int_{R}^\pi h(s) \Psi_{\gamma(s)} \, ds, \quad
h(s) := \frac{\cos R}{\sin^2 R} \sin s.
\end{align*}
\end{definition}

\begin{lemma}
\label{L4Zprod}
$-3 \tan R \left. \left \langle Z, \nabla \dbold_p\right \rangle \right|_{\partial B_R} =  
1+ \int_R^\pi h(s)\, ds.$
\end{lemma}
\begin{proof}
Fix $x\in \partial B_R$.  Using \ref{Lsincos} and  \ref{Lphidiv}.(ii),
\begin{multline*}
-3\sin R\int_{R}^\pi \left.\langle \Psi_{\gamma(s)}, \nabla \dbold_p\rangle\right|_x \sin s \, ds  
=
 -3\sin R\int_{R}^\pi \psi(\dbold_x)\left. \langle \nabla \dbold_{\gamma(s)}, \nabla \dbold_p \rangle \right|_x \sin s \, ds \\
 \begin{aligned}
&=  \int_R^\pi \left( \frac{1}{(1- \cos \dbold_x)^2}+ \frac{1}{1-\cos \dbold_x} \right)\left( \cos s - \cos R + \cos R(1-\cos \dbold_x)\right) \sin s \, ds \\
&=\int_R^\pi \left( \frac{\cos s - \cos R}{(1-\cos \dbold_x)^2} + \frac{\cos s}{1-\cos \dbold_x} + \cos R \right) \sin s \, ds\\
&= \int_R^\pi -\frac{\sin^2 s \frac{d}{ds}( 1- \cos \dbold_x)}{(1-\cos \dbold_x)^2} + \frac{2 \cos s \sin s}{1-\cos \dbold_x} + \cos R \sin s \, ds\\
&=\int_R^\pi \frac{d}{ds} \left( \frac{\sin^2 s}{1- \cos \dbold_x}\right) + \cos R \sin s \, ds \\
&= -\frac{\sin^2 R}{1-\cos r} + \cos R\int_R^\pi \sin s\, ds,
\end{aligned}
\end{multline*}
where the fourth equality uses the following rearrangement of \ref{Lsincos}.(iii): 
\begin{align*}
\cos s - \cos R = - \sin s \frac{d}{ds} ( 1- \cos \dbold_x)+ \cos s (1-\cos \dbold_x).
\end{align*}

By Lemma \ref{Lpsiprod}, 
\begin{align*}
- 3 \tan R \left. \langle\Psi_{y}, \nabla \dbold_p \rangle\right|_x  = \frac{1}{1-\cos r} + 1.
\end{align*}
Combining these calculations with Definition \ref{dW} finishes the proof.
\end{proof}

\subsection*{The 6 dimensional case}
Before defining $Z$, we need to derive a system of first order linear equations which specifies $h(s)$ when supplied with appropriate initial values. 

\begin{lemma}
\label{Lsys6}
The equation
\begin{align*}
-15 \sin R \, \left. \left\langle \Psi_{\gamma(s)}, \nabla \dbold_p\right\rangle\right|_x h(s) -3\cos R \, h(s)  = \frac{d}{ds} \sum_{i=1}^{2} \frac{f_i(s)\sin ^4 s }{(1-\cos \dbold_x)^i}
\end{align*}
is equivalent to the conditions that $h(s) = f_2(s) \sin^3 s$ and $\fbold: = (f_1, f_2)$ solves the system
\begin{align}
\label{EA}
A \fbold =  \fbold', \quad 
%\fbold := (f_1, f_2), \quad
A:= \frac{1}{\sin s} 
\left[\begin{array}{cc}-3 \cos s & 3 \cos s  \\ \cos R - \cos s  & \cos s-  \cos R\end{array}\right] .
\end{align}

\end{lemma}
\begin{proof}
Using Lemma \ref{Lphidiv}.(ii) and \ref{Lsincos}, 
\begin{equation}
\label{Ediff6}
\begin{aligned} 
-15 \sin R\left. \left\langle \Psi_{\gamma(s)} , \nabla \dbold_p\right\rangle\right|_{x} & = -15 \sin R\,  \psi(\dbold_x)\left. \left\langle \nabla \dbold_{\gamma(s)}, \nabla \dbold_p\right\rangle\right|_x\\ 
&=  \sum_{i=1}^3 \frac{b_i}{(1-\cos \dbold_x)^i}\left( \cos s - \cos R + \cos R(1-\cos \dbold_x)\right)\\
 &=  \sum_{i=1}^3 \frac{b_i(\cos s - \cos R)}{(1-\cos \dbold_x)^i } + \sum_{i=1}^3 \frac{b_i\cos R}{(1-\cos \dbold_x)^{i-1}}\\ 
 &= \frac{2(\cos s - \cos R)}{(1-\cos \dbold_x)^3} +  \frac{3 \cos s - \cos R}{(1-\cos \dbold_x)^2} +  \frac{3\cos s}{1-\cos \dbold_x}+ 3 \cos R,
\end{aligned}
\end{equation}
where $b_1 :=b_2 := 3,  b_3 := 2$,  recalling Remark \ref{Rjsmall}.
On the other hand, using \ref{Lsincos}.(iii), compute 
\begin{align*}
\frac{d}{ds} \sum_{i=1}^{2} \frac{f_i \sin^4 s}{(1-\cos \dbold_x)^i}&= 
 \sum_{i=1}^{2} \frac{(f_i \sin^4 s )'}{(1-\cos \dbold_x)^i} - \sum_{i=1}^{2} \frac{i f_i \sin^3 s}{(1-\cos \dbold_x)^{i+1}}(\cos R - \cos s + \cos s (1-\cos \dbold_x))\\
&=  \frac{2  (\cos s - \cos R)}{(1-\cos \dbold_x)^3}f_2 \sin^3 s + \frac{f'_2 \sin s + 2 f_2 \cos s - f_1  (\cos R- \cos s)}{(1-\cos \dbold_x)^2}\sin^3 s \\ 
&\mathrel{\phantom{=}}+ \frac{f'_1 \sin s + 3 f_1 \cos s}{1-\cos \dbold_x}\sin^3 s .
\end{align*}
The system \eqref{EA} follows from  multiplying \eqref{Ediff6} by $h(s)$ and matching coefficients above. 
\end{proof}

\begin{definition}
\label{D6}
Define a smooth vector field $Z$ on $B_R\setminus \{y \} $ by
\begin{equation*}
\begin{gathered}
 Z = \Psi_y + \int_R^\pi h(s) \Psi_{\gamma(s)} \, ds, 
 \end{gathered}
\end{equation*}
where $h(s) := f_2(s) \sin^3 s$ and $\fbold:=(f_1, f_2)$ is the solution of \eqref{EA} satisfying $\fbold(R) \sin^4(R) =\cos R (3, 2)$. 
\end{definition}

\begin{lemma} 
\label{L6Zprod}
$ -5\tan R  \left. \left\langle Z, \nabla \dbold_p\right \rangle\right|_{\partial B_R}=  1+  \int_R^\pi h(s)\, ds.$ 
\end{lemma}
\begin{proof}
By Lemma \ref{Lpsiprod}, 
\begin{align*}
-15 \tan R\left. \left\langle \Psi_y, \nabla \dbold_p \right\rangle\right|_{x} &=  \frac{2}{(1-\cos r)^2} + \frac{3}{1-\cos r} +3.
\end{align*}
Then Lemma \ref{Lsys6} implies 
\begin{multline*}
-15 \sin R \int_R^\pi h(s)\left.\langle \Psi_{\gamma(s)}, \nabla \dbold_p \rangle \right|_x\, ds - 3\cos R \int_R^\pi h(s) \, ds
= \int_R^\pi \frac{d}{ds}\bigg(\sum_{j=1}^2 \frac{f_j(s) \sin^4 s}{(1-\cos \dbold_x)^j} \bigg)\, ds \\
=- \sum_{j=1}^2 \frac{f_j(R) \sin^4 R}{(1- \cos r)^j}  
% &= 
%\frac{1}{15\sin R}\left( \frac{2\sin^2 R(1+\frac{1}{3} \sin^2 R)}{(1-\cos r)^2}+ \frac{\sin^2 R(3+\sin^2 R)}{1-\cos r}\right) - \frac{\cot R}{5}\int_R^\pi h(s)\, ds\\
=-\cos R \left(   \frac{2}{(1-\cos r)^2} + \frac{3}{1-\cos r}\right),
\end{multline*}
where the second equality uses that $\lim_{s \nearrow \pi} f_i(s)\sin^4 s = 0$, which follows either from standard ODE theory or the explicit formulae in Proposition \ref{Pexp}, and the last step follows from Definition \ref{D6}.  Combining these calculations proves the lemma. 
\end{proof}

\begin{lemma}
\label{Lhpos}
$h$ is nonnegative on $[R, \pi]$. 
\end{lemma}
\begin{proof}
It suffices to prove that $f_2$ is increasing on $(R, \pi)$.  From \eqref{EA}, $f_2$ satisfies the equation
\begin{align}
\label{Ef2}
(\sin s) f'_2 = (\cos R- \cos s) (f_1 - f_2). 
\end{align}
By inspection, the vector of constant functions $(1,1)$ solves \eqref{EA}.  
By the initial condition in \ref{D6},  $f_1(R)> f_2(R)$, so $\fbold$ is not a constant multiple of $(1,1)$.  Hence, by uniqueness of ODE solutions, $f_1-f_2>0$ on $(R, \pi)$.  Since $\cos R- \cos s>0 $ on $(R, \pi)$, it follows from \eqref{Ef2} that $f'_2> 0$ on $(R ,\pi)$.
\end{proof}

\subsection*{The general even dimensional case}
Assume $\Sigma$ has even dimension $k=2j$, where $j$ is at least $4$.

\begin{lemma}
\label{Lsystem}
The equation
\begin{align*}
 - \sin R \left. \left\langle \Psi_{\gamma(s)}, \nabla \dbold_p\right\rangle\right|_x h(s) - \frac{\cos R}{k-1} h(s)= \frac{d}{ds} \sum_{i=1}^{j-1} \frac{a_{i+1} f_i(s)\sin^{k-2} s }{(1-\cos \dbold_x)^i},
\end{align*}
is equivalent to a first order system $A \fbold =  \fbold'$, where $\fbold = (f_1, \dots, f_{j-1} )$.
\end{lemma}
\begin{proof}
One one hand, using Lemma \ref{Lphidiv}.(ii) and \ref{Lsincos},
\begin{align*} 
- \sin R \langle \Psi_{\gamma(s)} , \nabla \dbold_p\rangle & = -\sin R \, \psi(\dbold_x)\left.\left \langle \nabla \dbold_{\gamma(s)}, \nabla \dbold_p\right\rangle\right|_x\\ 
&= \sum_{i=1}^j \frac{a_i}{(1-\cos \dbold_x)^i}\left( \cos s - \cos R + \cos R(1-\cos \dbold_x)\right)\\ 
&= \sum_{i=1}^j \frac{a_i(\cos s - \cos R)}{(1-\cos \dbold_x)^i } + \sum_{i=1}^j \frac{a_i\cos R}{(1-\cos \dbold_x)^{i-1}}\\ 
&=  \frac{a_j (\cos s- \cos R)}{(1-\cos \dbold_x)^j}  + \sum_{i=1}^{j-1} \frac{a_i \cos s +(a_{i+1}-a_i) \cos R}{(1-\cos \dbold_x)^i} + a_1 \cos R\\
&= \frac{a_j (\cos s- \cos R)}{(1-\cos \dbold_x)^j}  + \sum_{i=1}^{j-1} \frac{a_i( \cos s +(1-i)c_i  \cos R)}{(1-\cos \dbold_x)^i} + a_1 \cos R,
\end{align*}
where the final step uses that $a_{i+1} - a_i  = \frac{1-i}{2j - (i+1)} a_i= (1-i)c_i a_i$ from the recursion formula in Lemma \ref{Lphidiv}.(ii).

On the other hand, using \ref{Lsincos}.(iii), compute 
\begin{multline*}
\frac{d}{ds} \sum_{i=1}^{j-1} \frac{a_{i+1} f_i \sin^{k-2} s}{(1-\cos \dbold_x)^i} =  \sum_{i=1}^{j-1} \frac{a_{i+1}(f_i \sin^{k-2} s )'}{(1-\cos \dbold_x)^i} - \sum_{i=1}^{j-1} \frac{i a_{i+1} f_i \sin^{k-3} s(\cos R - \cos s + \cos s (1-\cos \dbold_x))}{(1-\cos \dbold_x)^{i+1}}\\
\begin{aligned}
&= \frac{(j-1) a_{j} f_{j-1}\sin^{k-3} s (\cos s-\cos R)}{(1-\cos \dbold_x)^j} + \sum_{i=1}^{j-1} \frac{a_{i+1}  (f_i \sin^{k-2} s )'-i  a_{i+1}f_i \sin^{k-3} s\cos s }{(1-\cos \dbold_x)^i} \\ 
&\mathrel{\phantom{=}} -\sum_{i=2}^{j-1} \frac{(i-1) a_i f_{i-1}\sin^{k-3} s(\cos R -\cos s)}{(1-\cos \dbold_x)^{i}}\\
&= \frac{(j-1) a_{j} f_{j-1}\sin^{k-3} s (\cos s-\cos R)}{(1-\cos \dbold_x)^j}  \\
&\mathrel{\phantom{=}}+ \sum_{i=1}^{j-1}\sin^{k-3} s \frac{a_{i+1}(f'_i \sin s +(k-2-i)f_i \cos s) - a_i (i-1) f_{i-1} (\cos R - \cos s)}{(1-\cos \dbold_x)^i}. 
\end{aligned}
\end{multline*}

Matching coefficients on the terms over $(1-\cos \dbold_x)^j$ implies $h(s) = (j-1)f_{j-1} \sin^{k-3}s$.  Using this and matching coefficients in the other terms, we find for $i=1, \dots, j-1$
\begin{align*}
a_i\left( \cos s + (1-i)c_i \cos R\right) (j-1) f_{j-1} = a_{i+1}\left( f'_i \sin s + (k-2-i)f_i \cos s\right) - a_i (i-1)f_{i-1} (\cos R - \cos s).
\end{align*}
Solving each such equation for $f'_i$ establishes the system $A \fbold = \fbold'$ and completes the proof.
\end{proof}

\begin{definition}
\label{dZgen}
Define a vector field $Z$ on $B^n_R\setminus \{ y\}$ by .
\[ 
Z = \Psi_y+ \int_{R}^\pi h(s) \Psi_{\gamma(s)} \, ds,
\] 
where $h(s) = (j-1)f_{j-1}(s) \sin^{k-3}s$ and $\fbold$ is the solution of the system in \ref{Lsystem} satisfying $\fbold(R)\sin^{k-2}(R) = \cos R(1, \dots, 1)$.
\end{definition}

\section{Proof of Theorem \ref{Tmain}}
\label{S:Proof}
Assume now $k=2j$ and $j>1$, and let $Z$ be defined as in \ref{dW}, \ref{D6}, and \ref{dZgen}. 
\begin{definition}
\label{DWgen}
\begin{align*}
W = \frac{ \cos R \sin^{k-2}(R)}{(k-2) I_{k-2}(R)} \Phi_p + \frac{2I(R)}{I(\pi)} Z.
\end{align*}
\end{definition}

The following calculus identity (recall the notation of Definition \ref{DI}) will be useful:
\begin{align}
\label{Ecalculus}
I_{k}(R) =- \frac{1}{k-2} \cos R\sin^{k-2}(R) + \frac{k-2}{k-1} I_{k-2}(R).
\end{align}

\begin{remark}
\label{RWalt}
It will be useful to rewrite $W$  using \eqref{Ecalculus} as follows:
\begin{align*}
W = \left(1 -   \frac{k-1}{k-2}\frac{I(R)}{ I_{k-2}(R)} \right) \Phi_p + \frac{2I(R)}{I(\pi)} Z.
\end{align*}
\end{remark}

\begin{example}
\label{ExW}
When $k = 4$ and $k = 6$, calculations using \ref{DI} show that $W$ satisfies
\begin{align*}
W &= \cos R \cos^2(R/2)   \Phi_p +\frac{3}{2}I(R)  Z, \qquad (k=4)\\ 
W  &= \frac{3\cos^4(R/2)\cos R}{2+ \cos R} \Phi_p + \frac{15}{8} I(R)Z \qquad (k=6).
\end{align*}

\end{example}

\begin{lemma}[Constraint for $h$]
\label{Lch}
Let $Z$ and $h$ be as in Definition \ref{dW}, \ref{D6}, or \ref{dZgen}.  
\begin{enumerate}[label=\emph{(\roman*).}]
\item $\displaystyle{
1+ \int_R^\pi h(s)\, ds =\frac{k-1}{k-2} \frac{I(\pi)}{2} \frac{1}{I_{k-2}(R)} = \frac{I_{k-2}(\pi/2)}{I_{k-2}(R)}
}$.
\item $\displaystyle{\frac{2 I(R)}{I(\pi)} \left. \left \langle Z, \nabla \dbold_p \right \rangle\right|_{\partial B_R} =- \frac{\cot R}{k-2}\frac{I(R)}{I_{k-2}(R)}}.$
\end{enumerate}
\end{lemma}
\begin{proof}
In this proof, denote $C =  1+ \int_R^\pi h(s) \, ds$.  Let $\Sigma$ be a geodesic $k$-ball about $p$. 
As in the proof of \ref{Lsystem}, (see also \ref{L4Zprod} and \ref{L6Zprod}),
\begin{align*}
-\tan R \left.  \langle Z, \nabla \dbold_p\rangle\right|_{\partial B_R} = \frac{C}{k-1} \quad
\text{and} \quad
\ddiv_\Sigma Z = C. 
\end{align*}
Using the divergence theorem, we have for small $\varepsilon >0$
\begin{align*}
C \left| \Sigma \setminus B_{\varepsilon}(y) \right| = \int_{\Sigma \setminus B_{\varepsilon}(y)} \ddiv_\Sigma Z = 
\int_{\partial \Sigma \setminus B_\varepsilon(y)} \langle Z, \eta\rangle + \int_{\Sigma \cap \partial B_\varepsilon(y)} \langle Z , \eta \rangle.
\end{align*}
Letting $\varepsilon \searrow 0$, we find (recall Remark \ref{Rvol})
\begin{align*}
\omega C I(R)  = - \frac{\omega C}{k-1} \cos R \sin^{k-2}(R)  + \lim_{\varepsilon \searrow 0} \int_{\Sigma \cap \partial B_\varepsilon(y)} \langle Z , \eta \rangle.
\end{align*}

Arguing as in the proof of Proposition \ref{Pproof} and using \eqref{Ecalculus} to simplify, we find
\begin{align*}
C \frac{k-2}{k-1}I_{k-2}(R) = \frac{I(\pi)}{2}
\end{align*}
and the conclusion follows after simplifying and using \eqref{Ecalculus}.
\end{proof}

\begin{lemma}
When $k=4$ and $k=6$, $W$ satisfies  Proposition \ref{Pproof}.(i)-(iii).
\end{lemma}
\begin{proof}
It follows from \ref{Lphidiv}.(ii) that as $r: = \dbold_y \searrow 0$, 
\begin{align*}
\Psi_y =  - I(\pi) r^{1-k} \nabla r + o \left( r^{1-k}\right). 
\end{align*}
By Lemma \ref{LZsing}, the integral term in $Z$ contributes a singularity of order $o\left( r^{1-k}\right)$ as $r \searrow 0$.  (i) follows from combining these facts with the definition of $W$.  For (ii), compute using Definition \ref{Dphi} and Lemma \ref{Lch}.(ii)
\begin{align*}
\left.\left\langle W, \nabla \dbold_p\right\rangle\right|_{\partial B_R} &= \frac{\cos R \sin^{k-2}(R)}{(k-2)I_{k-2}(R)} \left. \left\langle \Phi_p , \nabla \dbold_p \right\rangle\right|_{\partial B_R} + \frac{2I(R)}{I(\pi)}\left.\left\langle Z, \nabla \dbold_p\right\rangle\right|_{\partial B_R} 
%&= \frac{\cot R}{k-2} \frac{I(R)}{I_{k-2}(R)} -  \frac{\cot R}{2j-2} \frac{I(R)}{I_{k-2}(R)}\\
=0.
\end{align*}
For (iii), calculate using \ref{Lphidiv}.(i) and \ref{Lch}.(ii) and Remark \ref{RWalt}
\begin{align*}
\ddiv_\Sigma W &= \left(1 -  \frac{k-1}{k-2}\frac{I(R)}{ I_{k-2}(R)} \right)\ddiv_\Sigma  \Phi_p + \frac{2I(R)}{I(\pi)} \ddiv_\Sigma Z\\
&\leq1 -  \frac{k-1}{k-2}\frac{I(R)}{ I_{k-2}(R)} + \frac{k-1}{k-2}\frac{I(R)}{ I_{k-2}(R)}\\
&= 1, 
\end{align*}
where before applying Lemma \ref{Lch}.(ii) we have used that $\ddiv_\Sigma Z \leq 1+ \int_R^\pi h(s) \, ds$, which uses that $h\geq 0$ (via Definition \ref{dW} when $k=4$ and Lemma \ref{Lhpos} when $k=6$) in conjunction with Lemma \ref{Lphidiv}.(ii).
\end{proof}

\begin{remark}
\label{Rhpos}
To prove the generalization of Theorem \ref{Tmain} in dimension $k = 2j$, $j\geq 4$ using the method above, it would suffice to prove that $h$ (recall Definition \ref{dZgen}) is nonnegative on $[R, \pi)$.  When $k = 8$, numerical calculations suggest that $h$ is not strictly nonnegative for certain values of $R$ and the method appears to break down.
\end{remark}

\begin{remark}
\label{RPsi3}
When $k = 3$, calculations using Definition \ref{Dphi} and Lemma \ref{Lphidiv}.(ii) show that
\[ 
\psi(r) = \frac{ r - \sin r \cos r - \pi}{2 \sin^2 r} \quad \text{and} \quad \tan R\left. \langle \Psi_y, \nabla \dbold_p\rangle \right|_x = \frac{r- \sin r \cos r - \pi}{ 2 (1+\cos r)\sin r}.  
\] 
These expressions should be contrasted with their even dimensional counterparts, respectively the formula in the second part of \ref{Lphidiv}.(ii) and the statement of Lemma \ref{Lpsiprod}.  When $k=2j$ is even, $h$ is defined as in \ref{dW} , \ref{D6}, and \ref{dZgen} so that the $r$ dependent terms in $\left. \langle \Psi_y, \nabla \dbold_p \rangle \right|_x$ are cancelled after adding $ \langle \int_R^\pi h(s) \Psi_{\gamma(s)}ds, \nabla \dbold_p \rangle$.   It would be interesting to know to define $h$ when $k$ is odd to produce the analogous cancellation.
\end{remark}

%%%%%%%%%%%%%%%%%%%%%%%%%%%%%%%%%%%%%%%%%%%%%%
%% APPENDIX
\begin{appendices}
\section{}

The system \eqref{EA} can be solved explicitly, and we sketch the details below for completeness.

\begin{prop}
\label{Pexp}
A matrix of solutions for \eqref{EA} is 
\begin{align*}
\phibold(s) = \left[\begin{array}{cc}1 & (1- \cos R \cos s + \cos^2 s) \csc^2 s  \left( \cot \frac{s}{2}\right)^{\cos R} \\    1 &( \cos^2 s -\cos R \cos s - \frac{1}{3} \sin^2 R)\csc^2 s  \left( \cot \frac{s}{2}\right)^{\cos R}\end{array}\right].
\end{align*}
Moreover, $h$ (recall Definition \ref{D6}) satisfies
\begin{align*}
h(s) = \cunder \left\{ (1+2\csc^2 R) \sin^3 s + 
 ( \cos^2 s -\cos R \cos s - \frac{1}{3} \sin^2 R)\sin s  \left( \frac{\cot(s/2)}{\cot(R/2)}\right)^{\cos R}\right\},
\end{align*}
where $\cunder: =  (3\cos R \csc^2 R)/(1+3\sin^2 R)$.
\end{prop}
\begin{proof}
Observe that $\fbold = (1,1)$ solves $\eqref{EA}$, and let $\phibold = ( (1, 1), (f_1, f_2))$ be a matrix of solutions, where $f_1$ and $f_2$ are to be determined.  
Liouville's formula implies
\begin{align}
\label{Eliouville1}
\det(  \phibold(s) ) = f_2 - f_1 = C \exp\left(\int \trace A(s) ds\right) = C \csc^2 s \left( \cot \frac{s}{2}\right)^{\cos R}.
\end{align}

Solving for $f_2$ in \eqref{Eliouville1} and substituting into the first item of \eqref{EA} implies
\begin{align*}
f_1' = 3 C \cos s \csc^3 s \left( \cot \frac{s}{2}\right)^{\cos R},
\end{align*}
which after integrating gives a solution $f_1$ of the form
\begin{align*}
f_1 &= \frac{-3C}{3+\sin^2 R}\csc^2 s \left( 1-  \cos R \cos s + \cos^2s \right)  \left( \cot \frac{s}{2}\right)^{\cos R}.
\end{align*}
Taking  $C = - 1-\frac{1}{3}\sin^2 R$ and substituting back into \eqref{Eliouville1} solves for $f_2$ and completes the proof of the formula for $\phibold$.  Next, define $\fbold = (f_1, f_2)$ by 
 \begin{align}
 \label{Efb}
\fbold = \cunder\,  \phibold \cdot \left[ \begin{array}{c}1+ 2\csc^2 R \\ ( \tan \frac{R}{2})^{\cos R}\end{array}\right], 
\end{align}
where $\cunder$ is as in the statement of the proposition.  A straightforward but omitted calculation shows that this solution of \eqref{EA} satisfies the initial values in Definition \ref{D6}. 
\end{proof}

\begin{remark}
The integral of $h$ may be computed directly, using that
\begin{align*} 
\int ( \cos^2 s -\cos R \cos s - \frac{1}{3} \sin^2 R)\sin s  \left( \cot \frac{s}{2}\right)^{\cos R}\, ds = 
-\frac{1}{3}(\cos R - \cos s )\sin^2 s  \left( \cot \frac{s}{2}\right)^{\cos R}.
\end{align*}
\end{remark}

Finally, by taking $R\searrow 0$ in combination with a rescaling argument, we show below that the proof of Theorem \ref{Tmain} recovers the Euclidean area bounds in \cite[Theorem 4]{Brendle:area} in dimensions $k=2, 4$ and $6$.

\begin{definition}
\label{Dexp}
Given $R\in (0, \pi)$, define $\Rcap : T_p \Sph^n \rightarrow T_p \Sph^n$ by $\Rcap v =  R v$, a magnified metric $\gtilde = \gtilde[R]$ on $B_R \subset \Sph^n$  and a metric $g_R$ on $B_1(0) \subset T_p \Sph^n$ by 
\[ 
\gtilde = R^{-2}g, \qquad
g_R : = \left( \exp_p \circ\,  \Rcap \right)^* \gtilde.
\] 
Denote by $\nablatilde$ and $\dboldtilde$ the Levi-Civita connection and the distance function induced by the metric $\gtilde$. 
\end{definition}
By Definition \ref{Dexp}, $\exp_p \circ \, \Rcap: (B_1(0), g_R) \rightarrow (B_R(p), \gtilde)$ is an isometry which we use to identify the two spaces.  Note that as $R\searrow 0$, $g_R$ converges smoothly to the euclidean metric $\left. g\right|_p$.  Using the identification above, we shall abuse notation by referring to $y$ both as a point on $\partial B_1(0)$ as well as a point on $\partial B_R(p) \subset \Sph^n$.

\begin{prop}[Euclidean asymptotics] As $R\searrow 0$, the vector fields $W$ in \cite[Definition 2.11]{FM} when $k=2$ and in \ref{DWgen} when $k =4, 6$ converge smoothly to fields $W_0$ on the euclidean ball $(B: =B_1(0), \left. g\right|_p)$ (using the notation above) satisfying
\label{Peuclid}
\begin{enumerate}[label=\emph{(\roman*).}]
\item $W_0 = -\frac{2}{k} \dbold^{1-k}_y \nabla \dbold_y$ as $\dbold_y \searrow 0$. 
\item $W_0$  is tangent to $\partial B$ along $\partial B\setminus \{ y\} $.
\item  $\ddiv_\Sigma W_0\leq  1$  for any minimal surface $\Sigma^k \subset B$, with  equality only  if $\nabla^\Sigma \dbold_p   = \nabla  \dbold_p$  on $\Sigma$.
\end{enumerate}
With these conditions, an appropriately modified version of Proposition \ref{Pproof} implies the area bounds in the euclidean setting (see the proof of \cite[Theorem 4]{Brendle:area}).
\end{prop}
In the proof we show slightly more: the limit $W_0$ is the field $W$ defined in \cite{Brendle:area}, up to a factor $2/k$.
\begin{proof}
Let $q\in B_p(R)$ and denote $\rtilde = \dboldtilde_q$.  Note that $\nablatilde \rtilde$ is a unit vector with respect to the $\gtilde$ metric.  By straightforward expansions using the definitions (recall \ref{Dphi}) we have
\begin{equation}
\label{Erescale}
  \begin{aligned}
\Phi_q &= \frac{\rtilde}{k} \nablatilde \rtilde + O(R^2)\\
2 \frac{I(R)}{I(\pi)} \Psi_q &= \frac{2}{k}  \left( -\rtilde^{1-k} + O(R)\right) \nablatilde \rtilde.
\end{aligned}
\end{equation}
Taking a limit as $R\searrow 0$, on the limit euclidean ball $(B_1(0), \left. g\right|_p)$,  $\Phi_q$ and $2 \frac{I(R)}{I(\pi)} \Psi_q$ converge to
\begin{align}
\label{Elimit}
\frac{x}{k} \quad \text{and} \quad -\frac{2}{k} \frac{x}{|x|^k},
\end{align}
where here $x$ is the position vector field on $B_1(0)$.

We first discuss the $k=2$ case.  The field $W$ from \cite[Definition 2.11]{FM} used to prove the two-dimensional area bound is $W = (\cos R) \Phi_p + (1-\cos R) \Psi_y$.  Noting that in this case $2 I(R)/I(\pi)  = 1- \cos R$, it follows from \eqref{Erescale} and \eqref{Elimit} by taking $R\searrow 0$, that on the limit euclidean ball $(B_1(0), \left. g\right|_p)$, $W$ converges to (in the notation of \cite{Brendle:area})
\[ 
\frac{x}{2} - \frac{x-y}{|x-y|^2},
\] 
which is the vector field used by Brendle in dimension 2.  (i)-(iii) can be checked either by passing to the limit from items (i)-(iii) of Proposition \ref{Pproof} or by direct calculation from the limit formula above. 

In dimensions $k = 4$ and $6$, the integral terms require some care.  Straightforward but tedious calculations using the explicit formulas for $h(s)$ in \ref{dW} and \ref{Pexp} when $k=4$ and $6$ show that as $R\searrow 0$, $2 \frac{I(R)}{I(\pi)} \int_R^\pi h(s) \Psi_{\gamma(s)}\, ds$ converges to
\begin{align*}
\frac{k-2}{k}\int_{1}^\infty u^{k-3} \frac{x - u y}{|x-uy|^k}\, du, 
\end{align*}
(cf. also \cite[Equation (1)]{Hitotumatu} and the discussion thereafter) which after the change of variable $u(t) = 1/t$ is equal to $\frac{k-2}{k} \int_0^1 \frac{tx - y}{|tx - y|^k} \, dt$. Using \ref{RWalt}, \eqref{Elimit} and the preceding, it follows that the limit $W_0$ is
	\begin{align*}
	W_0 = \frac{x}{k} - \frac{2}{k} \frac{x-y}{|x-y|^k} - \frac{k-2}{k}\int_0^1 \frac{tx-y}{|tx-y|^k}\, dt,
	\end{align*}
	which is up to a factor of $k/2$ the vector field defined in \cite{Brendle:area}.
\end{proof}

\end{appendices}
%\nocite{*}
\bibliographystyle{plain}

\end{document}